\documentclass[reqno,15pt]{amsart}
\usepackage{amsmath}
\usepackage{amssymb}
\usepackage{amscd}
\usepackage{graphics}
\usepackage{latexsym}
\usepackage{hyperref}

\theoremstyle{plain}
\newtheorem{theorem}{Theorem}[section]
\newtheorem{thm}[theorem]{Theorem}
\newtheorem{cor}[theorem]{Corollary}
\newtheorem{lem}[theorem]{Lemma}
\newtheorem{prop}[theorem]{Proposition}
\newtheorem{claim}[theorem]{Claim}

\theoremstyle{definition}
\newtheorem{defn}[theorem]{Definition}

\newtheorem{rmk}[theorem]{Remark}

\newtheorem{notat}[theorem]{Notation}

\theoremstyle{remark}

\input xy
\xyoption{all}

\newcommand{\marpar}[1]{}
\newcommand{\mni}{\medskip\noindent}

\newcommand{\PP}{\mathbb{P}}

\newcommand{\mc}{\mathcal}
\newcommand{\mf}{\mathfrak}

\newcommand{\OO}{\mc{O}}

\newcommand{\SP}{\text{Spec }}

\newcommand{\pis}[1]{\text{Pseudo}_{#1}}

\newsavebox{\sembox}
\newlength{\semwidth}
\newlength{\boxwidth}

\newcommand{\Sem}[1]{%
\sbox{\sembox}{\ensuremath{#1}}%
\settowidth{\semwidth}{\usebox{\sembox}}%
\sbox{\sembox}{\ensuremath{\left[\usebox{\sembox}\right]}}%
\settowidth{\boxwidth}{\usebox{\sembox}}%
\addtolength{\boxwidth}{-\semwidth}%
\left[\hspace{-0.3\boxwidth}%
\usebox{\sembox}%
\hspace{-0.3\boxwidth}\right]%
}

\newsavebox{\semrbox}
\newlength{\semrwidth}
\newlength{\boxrwidth}

\newcommand{\Semr}[1]{%
\sbox{\semrbox}{\ensuremath{#1}}%
\settowidth{\semrwidth}{\usebox{\semrbox}}%
\sbox{\semrbox}{\ensuremath{\left(\usebox{\semrbox}\right)}}%
\settowidth{\boxrwidth}{\usebox{\semrbox}}%
\addtolength{\boxrwidth}{-\semrwidth}%
\left(\hspace{-0.3\boxrwidth}%
\usebox{\semrbox}%
\hspace{-0.3\boxrwidth}\right)%
}

\def\Hilb{{Hilb}}

\title
{A local-global principle for weak approximation on varieties over
  function fields}

\author[Roth]{Mike Roth}
\address{
Department of Mathematics and Statistics \\
Queen's University \\ 
Kingston, Ontario, K7L 3N6}
\email{mikeroth@mast.queensu.ca}

\author[Starr]{Jason Michael Starr}
\address{Department of Mathematics \\
  Stony Brook University \\ Stony Brook, NY 11794}
\email{jstarr@math.sunysb.edu}

\date{\today}

\begin{document}


\begin{abstract}
We present a new perspective on the weak approximation conjecture of
Hassett and Tschinkel: formal sections of a rationally connected
fibration over a curve can be approximated to arbitrary order by
regular sections.  The new approach involves the study of how ideal
sheaves pullback to Cartier divisors.  
\end{abstract}


\maketitle



\section{Statement of results} \label{sec-intro}
\marpar{sec-intro}

\mni
\textbf{Disclaimer.}  This is the first draft of an article which was
circulated among experts and posted on the webpage of one of the
authors in January 2008.  A revision, with substantially stronger
results, is currently being drafted.  But as we have been informed
that these preliminary results are already useful in solving a related
problem of ``strong rational connectedness'' for surfaces, the authors
decided it was prudent to circulate this first draft more widely.

\mni
Let $k$ be an algebraically closed field.  Let $B$ be a smooth,
projective, connected curve over $k$.  

\begin{notat} \label{notat-rings}
\marpar{notat-rings}
For every closed point $b$ of
$B$, denote the 
Henselian local ring of $B$ at $b$ by $\OO_{B,b}^h$ and
denote its fraction field $K(\OO_{B,b}^h)$ by $K_{B,b}^h$.  Denote the
completion of the local ring by $\widehat{\OO}_{B,b}$, and denote its
fraction field by $\widehat{K}_{B,b}$.  

\mni
More generally, for an effective Cartier
divisor $E$ in $B$ with $\text{Supp}(E) =\{b_1,\dots,b_N\}$, denote by
$\OO_{B,E}$ the semilocal ring of $B$ at $\{b_1,\dots,b_N\}$, denote
by $\OO_{B,E}^h$ the product 
$\OO_{B,b_1}^h \times \dots \times \OO_{B,b_N}^h$, and
denote its total ring of fractions of $\OO_{B,E}^h$ by $K_{B,E}^h =
K_{B,b_1}^h \times \dots \times K_{B,b_N}^h$.  Denote by
$\widehat{\OO}_{B,E}$ the product $\widehat{\OO}_{B,b_1} \times \dots
\times \widehat{\OO}_{B,b_N}$, and denote by $\widehat{K}_{B,E}$ its
total ring of fractions $\widehat{K}_{B,E} \times \dots \times
\widehat{K}_{B,E}$.  
\end{notat}

\mni
For each of the $B$-algebras, $\OO = \OO_{B,E}$,
resp., $\OO = \OO_{B,E}^h$, $\OO = \widehat{\OO}_{B,E}$, the morphism
$$
\SP \OO \times_B E \rightarrow E
$$
is an isomorphism.  Thus we identify $E$ with its inverse image in
each of these affine $B$-schemes.    

\mni
Let $X$ be a normal, 
projective, connected $k$-scheme and let $\pi:X\rightarrow B$ be a
surjective $k$-morphism whose geometric generic fiber is normal.

\begin{defn} \label{defn-comblike}
\marpar{defn-comblike}
Let $s:B\rightarrow X$ be a section of $\pi$. 
A \emph{comb-like curve
  with handle $s$ in $X$} is a connected, closed curve $C$ in $X$ such
that $C\rightarrow X$ is a regular immersion
and such that the intersection of $C$ with the generic fiber of $B$
equals $s(\eta_B)$.
In other words, $s(B)$ is the only
irreducible component of $C$ dominating $B$, and $C$ is reduced at the
generic point of $s(B)$.  
\end{defn}

\begin{rmk} \label{rmk-comblike}
\marpar{rmk-comblike}
Among all comb-like curves are the \emph{combs} with handle $s$ as defined by
Koll\'ar-Miyaoka-Mori, and the \emph{combs with broken teeth} as defined by
Hassett-Tschinkel.  There are many comb-like curves which are neither
combs nor combs with broken teeth, most of which are
uninteresting.  We are interested only in those comb-like curves
which deform to smooth, connected curves.
This imposes many local constraints which we have not attempted to
characterize.  
\end{rmk}

\begin{defn} \label{defn-jet}
For an effective Cartier divisor $E$ in $B$, an \emph{$E$-section
  of $\pi$}, or just \emph{$E$-section}, 
is a $B$-morphism $s_E:E\rightarrow X$.  
A comb-like curve $C$ in $X$
\emph{deforms to section curves agreeing with $s_E$ Zariski locally},
resp. \emph{\'etale locally}, \emph{formally locally},
if for
$B' = \SP \OO_{B,E}$, resp. for $B'=\SP \OO_{B,E}^h$, $B'=\SP
\widehat{\OO}_{B,E}$, 
there exists a DVR $(R,\mf{n})$
with coefficient field $k$ 
and an $R$-flat, closed subscheme $\mc{C}\subset \SP R \times_{\SP k}
B' \times_B X$ 
such that
\begin{enumerate}
\item[(i)]
the fiber of $\mc{C}$ over the residue field
$\SP R/\mf{n}$ equals the base change of $C$ to
$B' \times_B X$, and
\item[(ii)]
the fiber of $\mc{C}$ over the fraction field $\SP K(R)$ equals the image
of a section $s_\eta:\SP K(R) \times_{\SP k} B' \rightarrow
\SP K(R) \times_{\SP k} B' \times_B X$ whose restriction to 
$\SP K(R) \times_{\SP k} E$ equals the
base change of $s_E$.  
\end{enumerate}
In other words, on a Zariski neighborhood, resp. 
\'etale neighborhood, formal neighborhood, of the fiber $\pi^{-1}(E)$
in $X$, $C'$ deforms to section curves all of which contain $s_E(E)$. 
\end{defn}

\begin{rmk} \label{rmk-jet}
\marpar{rmk-jet}
Let $E = E_1 \sqcup \dots \sqcup E_N$ be the connected components of
$E$.  Then $C$ deforms to section curves agreeing with $s_E$ \'etale
locally, resp. formally locally,  
if and only if for every $i=1,\dots,N$, $C$ deforms to section
curves agreeing with $s_E|_{E_i}$ \'etale locally, resp. formally
locally.  
Thus the \'etale and formal local conditions
can each be studied fiber-by-fiber.
\end{rmk}

\begin{thm} \label{thm-main}
\marpar{thm-main}
Let $s:B\rightarrow X$ be a section of $\pi$ mapping the geometric
generic point of $B$ into the very free locus of the geometric generic
fiber of $\pi$. (In particular, the very free locus is nonempty and
$\pi$ is a separably rationally connected fibration.)  
Let $E$ be an effective Cartier divisor in $B$.  Let
$s_E$ be an $E$-section of $\pi$.  Let $C$ be a comb-like curve with
handle $s$.  Assume $C$ deforms to section curves agreeing with $s_E$
\'etale locally or formally locally. 
Then $C$ deforms to section curves agreeing with
$s_E$ Zariski locally.  To be precise,
after attaching to $C$
sufficiently many very free teeth in fibers of $\pi$ at general points
of $s(B)$ and with general normal directions, $C$ deforms to comb-like
curves all of which are section
curves agreeing with $s_E$ over a Zariski open neighborhood of $E$ in $B$.    
\end{thm}

\mni
This reduces weak approximation to the problem of finding comb-like
curves $C$ with handle $s$ deforming to section curves agreeing with $s_E$
either \'etale locally or formally locally.  There is a criterion
due to Hassett guaranteeing existence of such curves.  
Originally formulated as a
criterion in characteristic $0$, the extension to positive
characteristic poses some minor issue in characteristics $2$, $3$ and
$5$.  In order to avoid these issues, we use a slight trick.

\begin{defn} \label{defn-conic}
\marpar{defn-conic}
Let $B$ be a smooth, connected, projective $k$-curve and let $E$ be an
effective Cartier divisor in $B$.  An \emph{$E$-conic} is a
smooth, projective $k$-surface $S$ together with a birational morphism
$\nu:P\rightarrow \PP^1\times_k B$ such that the restriction
$$
\nu:\nu{-1}(\PP^1\times_k(B-E)) \rightarrow \PP^1 \times_k (B-E)
$$
is an isomorphism.  The \emph{zero section} of $P$ is the unique
$B$-morphism
$$
0_P:B\rightarrow P
$$
whose composition with $\nu$ is the zero section of $\PP^1 \times_k
B$.  
\end{defn}

\begin{rmk} \label{rmk-conic}
\textbf{(i)}
Given a section $s$ and an $E$-section $s_E$, suppose there exists an
$E$-conic $\nu:P\rightarrow X\times_B \PP^1 $ and a comb-like curve
$C$ in $X\times_B \PP^1$ with handle $(s,0_P)$ which \'etale locally
or formally locally deforms to section curves agreeing with a given
$E$-section $(s_E,t_E)$.  Then Theorem ~\ref{thm-main} applies to
produce a section in $X\times_B \PP^1$ agreeing with $(s_E,t_E)$.
Thus the projection into $X$ is a section of $\pi$ agreeing with
$s_E$.  Therefore, to prove the existence of sections of $\pi$
agreeing with $s_E$, it suffices to prove the local condition for
$X\times_B \PP^1$

\mni
\textbf{(ii)}
Also,for a divisor $E$ with support $\{b_1,\dots,b_N\}$, 
the modification $\nu:P \rightarrow \PP^1\times_k B$ and the 
comb-like curve $C$ can both be defined in the neighborhood of every
fiber $\pi^{-1}(b_i)$ and then these local pieces can be assembled
into a global modification, resp. a global comb-like curve.  Thus the
existence of $\nu$ and $C$ with the required properties can be
verified near each fiber $\pi^{-1}(b_i)$ separately.
\end{rmk}

\begin{prop}[Hassett] \label{prop-Hassett}
\marpar{prop-Hassett}
Let $B$ be a smooth, connected, projective $k$-curve, let $X$ be a
normal, projective, connected $k$-scheme, and let $\pi:X\rightarrow B$
be a surjective $k$-morphism.  Let $s:B\rightarrow X$ be a section of
$\pi$ whose image $s(B)$ is contained in the smooth locus of $\pi$.
Let $E$ be an effective Cartier divisor in $B$ supported at a single
point $b$.  Let $\OO$ denote $\OO_{B,b}^h$, resp. $\widehat{\OO}_{B,b}$,
and let $K$ denote $K_{B,b}^h$, resp. $\widehat{K}_{B,b}$.  Let $s_E$ be an
$E$-section of $\pi$. 

\mni
Roughly stated, 
if the \'etale, resp. formal, germ of $s$ over $b$ gives a
$K$-point of $X(K)$ which is $R$-equivalent to the $K$-point of a germ
agreeing with $s_E$, then for some $E$-conic $P$, there exists a
comb-like curve $C$ in $X\times_B P$ with handle $(s,0_P)$ which
\'etale locally, resp. formally locally, deforms to section curves
agreeing with an $E$-section $(s_E,t_E)$.  

\mni
Following is one precise formulation (certainly not the most general).
Assume there exists a morphism of $K$-schemes
$$
f:\PP^1_{K} \rightarrow \SP K \times_B X
$$
whose closure $S$ in $\SP \OO \times_B X$ intersects the closed
fiber $\pi^{-1}(b)$ only at smooth points of $X$ and sending the
$K$-points $0$, resp. $\infty$, of $\PP^1_{K}(K)$ to the $K$-point of
the germ of $s$, 
resp. 
to a $K$-point whose unique extension
$$
s':\SP \OO \rightarrow \SP \OO \times_B X
$$
agrees with $s_E$.  Then there exists an $E$-conic $P$ and a comb-like
curve $C$ in $X\times_B P$ with handle $(s,0_P)$
which \'etale locally, resp. formally locally, deforms 
to section curves agreeing with an $E$-section $(s_E,t_E)$.

\mni
Moreover, if $f$ is an immersion and if the closed fiber of $S$ is
smooth, then there exists a comb-like curve $C$ in $X$ with handle $s$
which \'etale locally, resp. formally locally, deforms to section
curves agreeing with $s_E$ (i.e., there is no need for the $E$-conic).
\end{prop}

\mni
When combined with Proposition ~\ref{prop-Hassett}, Theorem
~\ref{thm-main} unifies and extends several known results regarding
weak approximation.  Most of these follow immediately; we content
ourselves with the following.

\begin{cor}\cite{HT06} \label{cor-HT}
\marpar{cor-HT}
With hypotheses as in Theorem ~\ref{thm-main}, assume also that $s(B)$ is
contained in the smooth locus of $X$.  Let $E$ be an effective Cartier
divisor in $B$ and let $s_E$ be an $E$-section of $\pi$.  If the
relative dimension of $\pi$ is $\geq 3$ and if every fiber
$\pi^{-1}(b)$ in $\pi^{-1}(E)$ is smooth and separably rationally connected,
then there exists a comb-like curve $C$
with handle $s$ which Zariski locally deforms to
section curves agreeing with $s_E$.  If the relative dimension is
$\leq 2$ and if every fiber $\pi^{-1}(b)$ in $\pi^{-1}(E)$ is smooth
and rationally connected, then for the $E$-conic $P = B\times_k \PP^1$
there exists a comb-like curve $C$ in $X\times_B P$ with
handle $(s,0_P)$ which Zariski locally deforms to section curves
agreeing with an $E$-section $(s_E,t_E)$.  

\mni
More generally, the two conclusions above
hold if for every point $b$ in $\text{Supp}(E)$ there exists an
immersed, resp. not necessarily immersed, very
free 
rational curve in the smooth locus of $\pi^{-1}(b)$ containing both
$s(b)$ and $s_E(b)$ (possibly both the same point).  
\end{cor}

\begin{cor} \label{cor-rational}
\marpar{cor-rational}
With hypotheses as in Theorem ~\ref{thm-main}, assume also that $s(B)$
is in the smooth locus of $X$.  Let
$E$ be an effective Cartier divisor in $B$ such that $X$ is smooth
along $\pi^{-1}(E)$.  Assume that for each $b$ in $\text{Supp}(E)$, either
$\SP \widehat{K}_{B,b} \times_B X$ is
$\widehat{K}_{B,b}$-rational or $\widehat{K}_{B,b}$-stably rational. 
Then for
every $E$-section $s_E$ there exists an $E$-conic $P$ over $B$ and a 
comb-like curve $C$ in $X\times_B P$ with
handle $(s,0_P)$ which Zariski locally deforms to section curves
agreeing with an $E$-section $(s_E,t_E)$.  
\end{cor}

\mni
In particular, for a smooth cubic surface over a local field
specializing to a cubic surface with only rational double points and
no more than $3$ of these, the cubic surface is rational.  

\begin{cor}\cite{HT07} \label{cor-cubic}
\marpar{cor-cubic}
Let $S$ be a smooth
cubic surface over $\SP k\Semr{t}$ extending to a flat, proper
scheme over $\SP k\Sem{t}$ whose closed fiber is a cubic surface
with only rational double
points, and at most $3$ of these.  Then $S$ is $k\Semr{t}$-rational.   

\mni
Assume $X$ is smooth and the geometric generic fiber of $\pi$ is a smooth
cubic surface.  Let $s$ be a section of $\pi$.  Let $E$ be a Cartier
divisor in $B$ such that for every point $b$ in $\text{Supp}(E)$ the
basechange $\SP \widehat{K}_{B,b}\times_B X$ extends to a flat, proper
scheme over $\SP \widehat{\OO}_{B,b}$ whose closed fiber is a cubic
surface with only rational double points, and at most $3$ of these.
Then for every $E$-section $s_E$, there exists an $E$-conic $P$
and a comb-like curve $C$ in $X\times_B P$ 
with handle $(s,0_P)$ which Zariski locally deforms 
to section curves agreeing with an $E$-section $(s_E,t_E)$.  
\end{cor}    

\begin{rmk} \label{rmk-cubic}
\marpar{rmk-cubic}
In fact Hassett and Tschinkel prove this result for families of cubic
surfaces whose singular fibers have rational double points, but
possibly more than $3$ such double points.  There is only one cubic
surface with $4$ rational double points, the zeroset of
$X_0X_1X_2+X_0X_1X_3 + X_0X_2X_3+X_1X_2X_3$ which has $4$ ordinary
double points at the vertices of the coordinate tetrahedron.  The
method in ~\cite{HT07} uses \emph{strong rational connectedness},
i.e., proving the smooth locus of each singular fiber is rationally
connected, combined with Corollary ~\ref{cor-HT}.  For the unique
cubic surface with $4$ ordinary double points, one can explicitly
construct rational curves in the smooth locus connecting any pair of
points.     
\end{rmk}


\section{Extension of two results of Grothendieck}
\label{sec-G}
\marpar{sec-G}

\mni
One of Grothendieck's results which is frequently useful in studying
coherent sheaves on proper schemes is \cite[Corollaire 7.7.8]{EGA3}.
The original proof includes a hypothesis which Grothendieck refers to
as ``\emph{surabondante}'' (Grothendieck's emphasis).  This hypothesis
has subsequently been removed, cf. ~\cite[Proposition
2.1.3]{Lieblich}.  We would like to briefly explain another way to
remove these hypotheses based on the representability of the Quot
scheme.  Using a general such representability 
result of Olsson, this leads to the following formulation. 

\begin{thm}\cite[Corollaire 7.7.8]{EGA3}, \cite[Proposition
  2.1.3]{Lieblich} \label{thm-Hom}
\marpar{thm-Hom}
Let $Y$ be an algebraic space.
Let $f:\mc{X}\rightarrow Y$ be a separated, locally finitely
presented, algebraic 
stack over the category $(\text{Aff}/Y)$ of affine $Y$-schemes.  Let
$\mc{F}$ and $\mc{G}$ be locally finitely presented, quasi-coherent 
$\OO_{\mc{X}}$-modules.  Assume that $\mc{G}$ is $Y$-flat and has
proper support over $Y$.  Then the covariant functor of quasi-coherent
$\OO_Y$-modules,
$$
\mc{T}(\mc{M}) :=
f_*(\textit{Hom}_{\OO_{\mc{X}}}(\mc{F},\mc{G}\otimes_{\OO_Y} \mc{M})),
$$
is representable by a locally finitely presented, quasi-coherent
$\OO_Y$-module $\mc{N}$, i.e., there is a natural equivalence of
functors
$$
\mc{T}(\mc{M}) \xrightarrow{\cong}
\textit{Hom}_{\OO_Y}(\mc{N},\mc{M}).
$$
\end{thm}  

\mni
Grothendieck deduced his version of this result as a corollary of
another theorem.  In fact the theorem also follows from the corollary.

\begin{cor}\cite[Th\'eor\`eme 7.7.6]{EGA3} \label{cor-Hom}
\marpar{cor-Hom}
Let $f:\mc{X} \rightarrow Y$ be as in Theorem ~\ref{thm-Hom}.  Let
$\mc{G}$ be a locally finitely presented, quasi-coherent
$\OO_{\mc{X}}$-module which is $Y$-flat and has proper support over
$Y$.  Then there exists a locally finitely presented, quasi-coherent
$\OO_Y$-module $\mc{Q}$ and a natural equivalence of covariant
functors of quasi-coherent $\OO_Y$-modules $\mc{M}$,
$$
f_*(\mc{G}\otimes_{\OO_Y}\mc{M}) \xrightarrow{\cong}
\textit{Hom}_{\OO_Y}(\mc{Q},\mc{M}).
$$
\end{cor}

\mni
Grothendieck in turn deduces \cite[Th\'eor\`eme 7.7.6]{EGA3} 
as a special case of a more general 
representability
result regarding the hyperderived pushforwards of a bounded below 
complex of locally
finitely presented, quasi-coherent
$\OO_{\mc{X}}$-modules which are $Y$-flat and have proper support over
$Y$.  Undoubtedly this general representability result extends to the
setting of stacks.  However, to deduce \cite[Corollaire 7.7.8]{EGA3}
from \cite[Th\'eor\`eme 7.7.6]{EGA3}, Grothendieck requires a
hypothesis that $\mc{F}$ has a presentation by locally free
$\OO_{\mc{X}}$-modules.  This is quite a restrictive hypothesis.

\mni
Instead one can try to prove Theorem ~\ref{thm-Hom} directly using
Artin's representability theorems.  This is precisely what Lieblich
does in ~\cite[Proposition 2.1.3]{Lieblich}.  Here we point
out that Theorem ~\ref{thm-Hom} also follows from
from representability of the Quot functor as proved by Olsson,
~\cite[Theorem 1.5]{OProper}.  

\begin{lem} \label{lem-Hom1}
\marpar{lem-Hom1}
Let $f:\mc{X} \rightarrow Y$ be as in Theorem ~\ref{thm-Hom}.  Let
$\mc{F}$ and $\mc{G}$ be locally finitely presented, quasi-coherent
$\OO_{\mc{X}}$-modules.  And let 
$$
\phi:\mc{F} \rightarrow \mc{G}
$$ 
be a
homomorphism of $\OO_{\mc{X}}$-modules.  
Assume $\mc{G}$ has proper support over $Y$,
resp. $\mc{F}$ and $\mc{G}$ have proper support over $Y$ and $\mc{G}$
is $Y$-flat.  Then there exists an open subspace $U$, resp. $V$, of
$S$ with the following property.  For every morphism $g:Y'\rightarrow
Y$ of algebraic spaces, the pullback morphism of sheaves on
$Y'\times_Y \mc{X}$, 
$$
(g,\text{Id}_{\mc{X}})^* \phi:
(g,\text{Id}_{\mc{X}})^* \mc{F} \rightarrow (g,\text{Id}_{\mc{X}})^*
\mc{G}
$$
is surjective, resp. an isomorphism, if and only if $g$ factors
through $U$, resp. $V$.
\end{lem}

\begin{proof}
Since $\text{Coker}(\phi)$ is a locally finitely presented,
quasi-coherent sheaf supported on $\text{Supp}(\mc{G})$,
$\text{Supp}(\text{Coker}(\phi))$ is closed in 
$\text{Supp}(\mc{G})$.  
Since
$\text{Supp}(\mc{G})$ is proper over $Y$, 
$f(\text{Supp}(\text{Coker}(\phi)))$ is a
closed subset of 
$Y$.  Define $U$ to be the open complement.  For a morphism
$g:Y'\rightarrow Y$, $(g,\text{Id}_{\mc{X}})^* \phi$ is surjective if
and only if $\text{Coker}((g,\text{Id}_{\mc{X}})^* \phi)$ is zero,
i.e., if and only if the support of
$\text{Coker}((g,\text{Id}_{\mc{X}})^* \phi)$ is empty.
Formation of the cokernel
is compatible with pullback, i.e.,
$$
\text{Coker}((g,\text{Id}_{\mc{X}})^* \phi) \cong
(g,\text{Id}_{\mc{X}})^*\text{Coker}(\phi). 
$$
Thus the support of $\text{Coker}((g,\text{Id}_{\mc{X}})^* \phi)$ is
empty if and only if $g(Y')$ is disjoint from
$f(\text{Supp}(\text{Coker}(\phi)))$, i.e., if and only if $g(Y')$ is 
contained in $U$.  

\mni
Next assume that $\mc{F}$, $\mc{G}$ each have proper support over $Y$
and $\mc{G}$ is $Y$-flat.  
Since $\mc{G}$ is $Y$-flat, the kernel of $\phi$ is locally
finitely presented by \cite[Lemme 11.3.9.1]{EGA4} (the property of
being locally finitely presented can be checked locally in the fppf
topology on $\mc{X}$, thus reduces to the case of a morphism of
schemes).  
The support of the
kernel is contained in the support of the kernel of $\mc{F}$.  Thus
$\text{Ker}(\phi)$ is a locally finitely presented, quasi-coherent
sheaf on the support of $\mc{F}$, which is proper over $Y$.  So the
support of $\text{Ker}(\phi)$ is also proper over $Y$.  Thus its image
under $f$ is a closed subset of $Y$.  Define $V$ to be the open
complement of this closed subset in $U$.

\mni
For every morphism $g:Y'\rightarrow U$, since $\mc{G}$ is $Y$-flat,
the following sequence is exact
$$
0 \rightarrow (g,\text{Id}_{\mc{X}})^*\text{Ker}(\phi) \rightarrow
(g,\text{Id}_{\mc{X}})^* \mc{F} \rightarrow (g,\text{Id}_{\mc{X}})^*
\mc{G} \rightarrow 0.
$$
Thus $(g,\text{Id}_{\mc{X}})^*\phi$ is an isomorphism if and only if
$(g,\text{Id}_{\mc{X}})^* \text{Ker}(\phi)$ is zero, i.e., if and only
if $g(Y')$ is contained in $V$.
\end{proof}

\begin{lem} \label{lem-Hom2}
\marpar{lem-Hom2}
Let $f:\mc{X} \rightarrow Y$, $\mc{F}$ and $\mc{G}$ be as in Theorem
~\ref{thm-Hom}.  There exists a locally finitely presented, separated
morphism of algebraic spaces $h:Z\rightarrow Y$ and a morphism of
quasi-coherent sheaves on $Z\times_Y \mc{X}$ 
$$
\phi:(h,\text{Id}_{\mc{X}})^*\mc{F} \rightarrow
(h,\text{Id}_{\mc{X}})^* \mc{G}
$$
which represents the contravariant functor associating to every
morphism $g:Y'\rightarrow Y$ the set of morphisms of quasi-coherent
sheaves on $Y'\times_Y \mc{X}$
$$
\psi:(g,\text{Id}_{\mc{X}})^*\mc{F} \rightarrow
(g,\text{Id}_{\mc{X}})^* \mc{G}.
$$
\end{lem}

\begin{proof}
By \cite[Theorem 1.5]{OProper}, there exists a locally finitely
presented, separated morphism $i:W\rightarrow Y$ of algebraic spaces
and a quotient
$$
\theta:(i,\text{Id}_{\mc{X}})^*(\mc{F}\oplus \mc{G}) \rightarrow \mc{H}
$$
representing the Quot functor of flat families of locally finitely
presented, quasi-coherent quotients of the pullback of $\mc{F}\oplus
\mc{G}$ having proper support over the base.  Denote by
$$
\theta_{\mc{G}}:(i,\text{Id}_{\mc{X}})^* \mc{G} \rightarrow \mc{H}
$$
the composition of the summand
$$
e_{\mc{G}}:(i,\text{Id}_{\mc{X}})^* \mc{G} \hookrightarrow
(i,\text{Id}_{\mc{X}})^*(\mc{F}\oplus \mc{G})
$$ 
with $\theta$.  

\mni
By Lemma ~\ref{lem-Hom1}, there is an open subspace $Z$ of $W$ such
that for every morphism $j:Y'\rightarrow W$, the pullback
$$
(j,\text{Id}_{\mc{X}})^* \theta_{\mc{G}}: (i\circ
j,\text{Id}_{\mc{X}})^* \mc{G} \rightarrow (j,\text{Id}_{\mc{X}})^*
\mc{H}
$$
is an isomorphism if and only if $j(Y')$ factors through $Z$.  
Denote by $h:Z\rightarrow Y$ the restriction of $i$ to $Z$.  
If $j(Y')$ factors through $Z$, then
$(j,\text{Id}_{\mc{X}})^* \theta$ equals the composition
$$
(i\circ j,\text{Id}_{\mc{X}})^*(\mc{F}\oplus \mc{G})
\xrightarrow{(\psi,\text{Id})} (i\circ j,\text{Id}_{\mc{X}})^*\mc{G}
\xrightarrow{j^*\theta_{\mc{G}}} (j,\text{Id}_{\mc{X}})^*\mc{H}
$$
for a unique morphism of quasi-coherent sheaves
$$
\psi:(i\circ j,\text{Id}_{\mc{X}})^* \mc{F} \rightarrow (i\circ
j,\text{Id}_{\mc{X}})^* \mc{G}.
$$
In particular, applied to $\text{Id}_Z:Z\rightarrow Z$, this produces
the homomorphism $\phi$.  And it is straightforward to see that the
natural transformation associating to every morphism $j:Y'\rightarrow
Z$ the homomorphism $\psi$ is an equivalence of functors, i.e.,
$(h:Z\rightarrow Y,\phi)$ is universal.
\end{proof}

\begin{proof}[Proof of Theorem ~\ref{thm-Hom}]
Let $h:Z\rightarrow Y$ and $\phi$ be as in Lemma ~\ref{lem-Hom2}.
The zero homomorphism $0:\mc{F} \rightarrow \mc{G}$
defines a $Y$-morphism $z:Y \rightarrow Z$.  Since $Z$ is separated
over $Y$, the $Y$-morphism $z$ is a closed immersion.  Since $Z$ is
locally finitely presented over $Y$, the ideal sheaf $\mc{I}$ of this
closed immersion is a locally finitely presented, quasi-coherent 
$\OO_Z$-module.
Thus $z^*\mc{I}$ is a locally finitely presented, quasi-coherent 
$\OO_Y$-module.  Denote this $\OO_Y$-module by $\mc{N}$.  

\mni
Consider the closed subspace $Z_1$ of $Z$ with ideal sheaf
$\mc{I}^2$.  The restriction of $h$ to $Z_1$ is a finite morphism,
thus equivalent to the locally finitely presented $\OO_Y$-algebra $h_*
\OO_{Z_1}$.  Of course this fits into a short exact sequence
$$
0 \rightarrow z^* \mc{I} \rightarrow h_*\OO_{Z_1} \rightarrow \OO_Y
\rightarrow 0
$$
where the injection is an ideal sheaf and the surjection is a
homomorphism of $\OO_Y$-algebras.  And the morphism $z$ defines a
splitting of this surjection of $\OO_Y$-algebras.  Thus, as an
$\OO_Y$-algebra, there is a canonical isomorphism
$$
h_*\OO_{Z_1} \cong \OO_Y \oplus z^* \mc{I}.
$$
The restriction of $\phi$ to $Z_1$ together with
adjunction of $h^*$ and $h_*$ defines a homomorphism of $\OO_Y$-modules
$$
\mc{F}\rightarrow \mc{G}\otimes_{\OO_Y} h_*\OO_{Z_1}.
$$
Using the canonical isomorphism, this homomorphism is of the form
$$
\mc{F} \xrightarrow{(0,\chi)} \mc{G} \oplus (\mc{G}\otimes_{\OO_Y} \mc{N}).
$$
The homomorphism $\chi$ defines a natural transformation of covariant
functors of quasi-coherent $\OO_Y$-modules $\mc{M}$
$$
\textit{Hom}_{\OO_Y}(\mc{N},\mc{M}) \rightarrow
f_*(\textit{Hom}_{\OO_{\mc{X}}}(\mc{F}, \mc{G}\otimes_{\OO_Y}\mc{M}).
$$
By the same argument used to prove the equivalence of parts (a) and
(d) of \cite[Th\'eor\`eme 7.7.5]{EGA3}, this is an equivalence of
functors and the induced $Y$-morphism
$$
\underline{\text{Spec}}_Y(\text{Sym}^\bullet(\mc{N})) \rightarrow Z
$$
is an isomorphism.  
\end{proof}

\begin{proof}[Proof of Corollary ~\ref{cor-Hom}]
This follows from Theorem ~\ref{thm-Hom} by taking $\mc{F}$ to be
$\OO_{\mc{X}}$.  
\end{proof}


\section{Pseudo-ideal sheaves} \label{sec-pis}
\marpar{sec-pis}

\mni
Let $f:X \rightarrow Y$ be a flat, locally finitely presented,
proper morphism of algebraic spaces.  For every morphism of algebraic
spaces, $g:Y'\rightarrow Y$, denote by $f_{Y'}:X_{Y'}\rightarrow Y'$ the
basechange of $f$.  Using results of Martin Olsson,
the following definitions and results are still valid whenever $X$ is
a flat, locally finitely presented, proper algebraic stack over $Y$.
But since our application is to algebraic spaces, we leave the case of
stacks to the interested reader.

\begin{defn} \label{defn-pis}
\marpar{defn-pis}
For every morphism $g:Y'\rightarrow Y$ of algebraic spaces, a flat
family of \emph{pseudo-ideal sheaves of $X/Y$ over $Y'$} is a pair
$(\mc{F},u)$ consisting of
\begin{enumerate}
\item[(i)]
a $Y'$-flat, locally finitely presented, quasi-coherent
$\OO_{X_{Y'}}$-module $\mc{F}$, and
\item[(ii)]
an $\OO_{X_{Y'}}$-homomorphism $u:\mc{F} \rightarrow \OO_{X_{Y'}}$
\end{enumerate}
such that the following induced morphism is zero,
$$
u':\bigwedge^2 \mc{F} \rightarrow \mc{F}, \ \ f_1\wedge f_2 \mapsto
u(f_1)f_2 - u(f_2)f_1.
$$

\mni
For every pair $g_1:Y'_1 \rightarrow Y$, $g_2:Y'_2 \rightarrow Y$ of
morphisms of algebraic spaces, for every pair $(\mc{F}_1,u_1)$, resp.
$(\mc{F}_2,u_2)$, of flat families of pseudo-ideal sheaves of $X/Y$
over $Y'_1$, resp. over $Y'_2$, and 
for every $Y$-morphism $h:Y'_1 \rightarrow Y'_2$, a \emph{pullback map} 
from $(\mc{F}_1,u_1)$ to $(\mc{F}_2,u_2)$ over $h$ is an
isomorphism of $\OO_{X_{Y'_1}}$-modules
$$
\eta:\mc{F}_1 \rightarrow h^* \mc{F}_2
$$
such that $h^* u_2 \circ \eta$ equals $u_1$.  

\mni
The \emph{category of pseudo-ideal sheaves of $X/Y$},
$\pis{X/S}$, 
is the category
whose objects are data $(g:Y'\rightarrow Y,(\mc{F},u))$ of an affine $Y$-scheme
$Y'$ together with a flat family of pseudo-ideal sheaves of $X/Y$ over
$Y'$, and whose Hom sets
$$
\text{Hom}((g_1:Y'_1\rightarrow Y,(\mc{F}_1,u_1)),(g_2:Y'_2\rightarrow
Y,(\mc{F}_2,u_2)))
$$
are the sets of pairs $(h,\eta)$ of a $Y$-morphism $h:Y'_1\rightarrow
Y'_2$ together with a pullback map $\eta$ from $(\mc{F}_1,u_1)$ to
$(\mc{F}_2,u_2)$ over $h$.  Identity morphisms and composition of
morphisms are defined in the obvious manner.  There is an obvious
functor from $\pis{X/Y}$ to the category $(\text{Aff}/Y)$ of
affine $Y$-schemes sending every object $(g:Y'\rightarrow Y,(\mc{F},u))$
to $(g:Y'\rightarrow Y)$ and sending every morphism $(h,\eta)$ to $h$.  
\end{defn}

\begin{prop} \label{prop-algebraic}
\marpar{prop-algebraic}
The category $\pis{X/Y}$ is a limit-preserving 
algebraic stack over the category $(\text{Aff}/Y)$ of affine $Y$-schemes.
Moreover, the diagonal is quasi-compact and separated.  
\end{prop}

\begin{proof}
Denote by $\text{Coh}_{X/Y}$ the category of coherent sheaves on
$X/Y$, cf. \cite[(2.4.4)]{LM-B}.  
There is a functor $G:\pis{X/Y} \rightarrow \text{Coh}_{X/Y}$ sending every
object
$(g:Y'\rightarrow Y,(\mc{F},u))$ to $(g:Y'\rightarrow Y,\mc{F})$ and
sending every morphism $(h,\eta)$ to $(h,\eta)$.  This is a $1$-morphism of
categories over $(\text{Aff}/Y)$.  
By \cite[Th\'eor\`eme
4.6.2.1]{LM-B} and \cite[Proposition 4.1]{Smoduli}, $\text{Coh}_{X/Y}$
is a limit-preserving algebraic stack over $(\text{Aff}/Y)$ with
quasi-compact, separated diagonal.  Thus to prove the proposition, it
suffices to prove that $G$ is representable by locally finitely
presented, separated algebraic spaces.

\mni
Let $Y'$ be a $Y$-algebraic space and let $\mc{F}$ be a locally
finitely presented, quasi-coherent
$\OO_{X_{Y'}}$-module.  Since $f_{Y'}:X_{Y'}\rightarrow Y'$ is flat,
locally finitely presented and proper, by Lemma ~\ref{lem-Hom2},
there exists a locally finitely presented, separated morphism
$h:Z\rightarrow Y'$ of algebraic spaces and a universal homomorphism
$$
u:(h,\text{Id}_{X_{Y'}})^* \mc{F} \rightarrow \OO_{X_Z}.
$$
In fact, as follows from the proof of Theorem ~\ref{thm-Hom}, there is
a 
locally finitely presented, quasi-coherent
$\OO_{Y'}$-module $\mc{N}$ and a homomorphism of
$\OO_{X_{Y'}}$-modules
$$
\chi:\mc{F} \rightarrow \OO_{X_{Y'}}\otimes_{\OO_{Y'}} \mc{N}
$$
such that $Z =
\underline{\text{Spec}}_{Y'}(\text{Sym}^\bullet(\mc{N}))$ and such
that $u$ is the homomorphism induced by $\chi$.  

\mni
Since $\mc{F}$ is a locally finitely presented, quasi-coherent
$\OO_{X_{Y'}}$-module which is $Y'$-flat and has proper support over
$Y'$, the same is true of $(h,\text{Id}_{X_{Y'}})^* \mc{F}$ relative
to $Z$.  Thus, by Theorem ~\ref{thm-Hom}, there exists a locally
finitely presented, quasi-coherent $\OO_Z$-module $\mc{N}$ and a
universal homomorphism
$$
\chi:\bigwedge^2 (h,\text{Id}_{X_{Y'}})^* \mc{F} \rightarrow
(h,\text{Id}_{X_{Y'}})^* \mc{F}\otimes_{\OO_Z} \mc{N}. 
$$
Thus there exists a unique homomorphism of $\OO_Z$-modules,
$$
\zeta:\mc{N} \rightarrow \OO_Z
$$
such that the induced homomorphism of $\OO_{X_Z}$-modules
$$
\bigwedge^2 (h,\text{Id}_{X_{Y'}})^* \mc{F} \xrightarrow{\chi}
(h,\text{Id}_{X_{Y'}})^* \mc{F}\otimes_{\OO_Z} \mc{N}
\xrightarrow{\text{Id}\otimes \zeta} (h,\text{Id}_{X_{Y'}})^* \mc{F}
$$
equals $u'$.  Let $P$ denote the closed subspace of $Z$ whose ideal
sheaf equals $\zeta(\mc{N})$.  Chasing diagrams, the restriction
morphism $h|_P:P\rightarrow Y'$ together with the pullback of $u$ to
$P$ is a pair representing $Y'\times_{\text{Coh}_{X/Y}} \pis{X/Y}
\rightarrow Y'$.  Since $Z\rightarrow Y'$ is a locally finitely
presented, separated morphism of algebraic spaces (also schematic),
and since $P$ is a closed subspace of $Y'$ whose
ideal sheaf is locally finitely generated (being the image of the
locally finitely presented sheaf $\mc{N}$), also $P \rightarrow Y$ is
a locally finitely presented, separated morphism of algebraic spaces
(also schematic).
\end{proof}

\mni
Denote by $(g:\text{Hilb}_{X/Y}\rightarrow Y,C)$ a universal pair of a
morphism of algebraic spaces $g$ and a closed subspace $C$ of
$\text{Hilb}_{X/Y}\times_Y X$ which is flat, locally finitely
presented, and proper over $\text{Hilb}_{X/Y}$, i.e., an object
representing the Hilbert functor, cf. \cite[Corollary 6.2]{Artin}.
Denote by 
$$
0 \rightarrow \mc{I} \xrightarrow{u} \OO_{\text{Hilb}_{X/Y}\times_Y X}
\rightarrow \OO_C \rightarrow 0.
$$
the natural exact sequence, where $\mc{I}$ is the ideal sheaf of $C$
in $\text{Hilb}_{X/Y}\times_Y X$.  

\begin{prop} \label{prop-Hilbopen}
\marpar{prop-Hilbopen}
The pair $(\mc{I},u)$ is a family of 
pseudo-ideal sheaves of $X/Y$ over $\text{Hilb}_{X/Y}$.  The induced
$1$-morphism
$$
\iota:\text{Hilb}_{X/Y} \rightarrow \pis{X/Y}
$$
is representable by open immersions.  
\end{prop}

\begin{proof}
Since the kernel of a surjection of flat modules is flat, $\mc{I}$ is
flat over $\text{Hilb}_{X/Y}$.  By \cite[Lemme 11.3.9.1]{EGA4},
$\mc{I}$ is a locally finitely presented, quasi-coherent sheaf.  Since
the homomorphism $u$ is injective, to prove that $u'$ is zero it
suffices to prove the composition $u\circ u'$ is zero.  This follows
immediately from the definition of $u'$.  
Thus
the pair $(\mc{I},u)$ is a family of pseudo-ideal sheaves of $X/Y$
over $\text{Hilb}_{X/Y}$.

\mni  
To prove that $\iota$ is representable by open immersions, it suffices
to prove that it is representable by quasi-compact, \'etale
monomorphisms of schemes.  To this end, let $Y'$ be an affine
$Y$-scheme and let $(\mc{F},v)$ be a flat family of pseudo-ideal
sheaves of $X/Y$ over $Y'$.  Denote the cokernel of $v$ by 
$$
w:\OO_{X_{Y'}} \rightarrow \mc{G}.
$$ 
By \cite[Theorem 3.2]{OS}, there is a morphism of algebraic spaces
$\sigma:\Sigma\rightarrow Y'$ such that
$(\sigma,\text{Id}_{X_{Y'}})^*\mc{G}$ 
is flat over $\Sigma$ and such that $\Sigma$ is universal among $Y'$-spaces
with this property.  Moreover, $\Sigma$ is a surjective,
finitely-presented, quasi-affine monomorphism (in particular
schematic).  As an aside, please note that the remark preceding
\cite[Theorem 3.2]{OS} is incomplete -- \cite[Proposition 3.1]{OS}
should be properly attributed to Laumon and Moret-Bailly,
\cite[Th\'eor\`eme A.2]{LM-B}.  

\mni
By \cite[Lemma 11.3.9.1]{EGA4},
$\text{Ker}((\sigma,\text{Id}_{X_{Y'}})^*w)$ is a locally finitely
presented, quasi-coherent sheaf.  Moreover, because it is the kernel
of a surjection of sheaves which are flat over $\Sigma$, it is also
flat over $\Sigma$.  
By Lemma ~\ref{lem-Hom1}, there is an open subscheme $W$ of $\Sigma$
such that a morphism $S\rightarrow \Sigma$ factors through $W$ if and
only if the pullback of
$$
(\sigma,\text{Id}_{X_{Y'}})^* \mc{F} \rightarrow 
\text{Ker}((\sigma,\text{Id}_{X_{Y'}})^*w)
$$
is an isomorphism.  Chasing universal properties, it is clear that $W
\rightarrow Y'$ represents
$$
Y' \times_{\pis{X/Y}} \text{Hilb}_{X/Y}.
$$
Thus $\iota$ is representable by finitely-presented, quasi-affine
monomorphisms of schemes.  

\mni
It only remains to prove that $\iota$ is \'etale.  Because $\iota$ is
a finitely-presented it remains to prove
that $\iota$ is formally \'etale.  Thus, let $Y'=\SP A'$ where $A'$ is
a local Artin $\OO_Y$-algebra with maximal ideal $\mathfrak{m}$ and residue
field $\kappa$.  
And let 
$$
0 \rightarrow J \rightarrow A' \rightarrow A \rightarrow 0
$$
be an infinitesimal extension, i.e., $\mathfrak{m}J$ is zero.
Let $(\mc{F},u)$ be a pseudo-ideal sheaf of
$X/Y$ over $Y'$, and assume the basechange to $\SP A$ is an ideal
sheaf with $A$-flat cokernel.  Since $\iota$ is a monomorphism,
formal \'etaleness for $\iota$
precisely says that $Y'\rightarrow \pis{X/Y}$ factors through $\iota$,
i.e., $u$ is injective and $\text{Coker}(u)$ is $A'$-flat.

\mni
To prove this use the local flatness criterion,
e.g., as formulated in ~\cite[Proposition 11.3.7]{EGA4}.  This
criterion is an equivalence between the conditions of
\begin{enumerate}
\item[(i)]
injectivity of $u$ is injective and 
$A'$-flatness of $\text{Coker}(u)$
\item[(ii)]
and injectivity of
$$
u\otimes_{A'}\kappa: \mc{F}\otimes_{A'}\kappa \rightarrow
\OO_{X_{Y'}}\otimes_{A'} \kappa
$$
\end{enumerate}
By hypothesis, (i) holds after
basechange to $A$.  Thus (ii) holds after
basechange to $A$.  But since $A'/\mf{m}$ equals $A/\mf{m}$, (ii) for
the original family over $A'$ is precisely the same as (ii) for the
basechange family over $A$.
Thus also (i) holds over $A'$.
\end{proof}

\mni
The significance of pseudo-ideal sheaves has to do with restriction to
Cartier divisors.  Let $D$ be an effective Cartier divisor in $X$,
considered as a closed subscheme of $X$, and assume $D$ is flat over
$Y$.  Denote by $\mc{I}_D$ the pullback 
$$
\mc{I}_D :=
\mc{I}\otimes_{\OO_X} \OO_D
$$ on $\text{Hilb}_{X/Y}\times_X D$.  And denote by
$$
u_D:\mc{I}_D \rightarrow \OO_{\text{Hilb}_{X/Y}\times_X D}
$$
the restriction of $u$.

\begin{prop} \label{prop-flat}
\marpar{prop-flat}
The locally finitely presented, quasi-coherent sheaf $\mc{I}_D$ is
flat over $\text{Hilb}_{X/Y}$.  Thus the pair $(\mc{I}_D,u_D)$ is a
flat family of pseudo-ideal sheaves of $D/Y$ over $\text{Hilb}_{X/Y}$. 
\end{prop}

\begin{proof}
Associated to the Cartier divisor $D$ there is an injective homomorphism of
invertible sheaves 
$$
t':\OO_X(-D) \xrightarrow \OO_X.
$$  
This induces a
morphism of locally finitely presented, quasi-coherent sheaves
$$
t:\mc{I}\otimes_{\OO_X} \OO_X(-D) \rightarrow \mc{I}.
$$
The cokernel of $t$ is $\mc{I}_D$.  
By the local flatness criterion, \cite[Proposition 11.3.7]{EGA4}, to
prove that $t$ is injective and $\mc{I}_D$ is flat over
$\text{Hilb}_{X/Y}$, it suffices to prove that the ``fiber'' of $t$
over every point of $\text{Hilb}_{X/Y}$ is injective.  Thus, let
$\kappa$ be a field, let $y:\SP \kappa Y$ be a morphism, and let
$\mc{I}_y \xrightarrow \OO_{X_y}$ be an ideal sheaf.  Since $D$ is
$Y$-flat, the homomorphism of locally free sheaves
$$
t'_y:\OO_X(-D)\otimes_{\OO_Y} \kappa \rightarrow \OO_X\otimes_{\OO_Y}
\kappa
$$
is injective, thus a flat resolution of $\OO_D\otimes_{\OO_Y} \kappa$.
In particular, $\textit{Tor}_2^{\OO_{X_y}}(\OO_{X_y}/\mc{I}_y,
\OO_{D_y})$ equals zero because there is a flat resolution of
$\OO_{D_y}$ with amplitude $[-1,0]$.  By the long exact sequence
of Tor associated to the short exact sequence
$$
0 \rightarrow \mc{I}_y \rightarrow \OO_{X_y} \rightarrow
\OO_{X_y}/\mc{I}_y \rightarrow 0,
$$
there is an isomorphism
$$
\textit{Tor}_1^{\OO_{X_y}}(\mc{I}_y,
\OO_{D_y}) \cong
\textit{Tor}_2^{\OO_{X_y}}(\OO_{X_y}/\mc{I}_y,
\OO_{D_y}) = 0.
$$
But this Tor sheaf is precisely the kernel of
$$
t_y:\mc{I}_y\otimes_{\OO_X} \OO_X(-D) \rightarrow \mc{I}_y.
$$
Thus $t_y$ is injective, and so $\mc{I}_D$ is flat over
$\text{Hilb}_{X/Y}$.  
\end{proof}

\begin{notat} \label{notat-iotaD}
\marpar{notat-iotaD}
Denote by 
$$
\iota_D:\text{Hilb}_{X/Y} \rightarrow \pis{D/Y}
$$
the $1$-morphism associated to the flat family $(\mc{I}_D,u_D)$
of pseudo-ideal sheaves of $D/Y$ over $\text{Hilb}_{X/Y}$.
This is the \emph{divisor restriction map}.  
\end{notat}

\mni
Since $\text{Hilb}_{X/Y}$ and $\pis{D/Y}$ are both locally finitely
presented over $Y$, $\iota_D$ is locally finitely presented.  Since
$\Hilb_{X/Y}$ is an algebraic space, $\iota_D$ is representable (by
morphisms of algebraic spaces).  Since the diagonal morphism of
$\pis{D/Y}$ over $Y$ is separated, and since $\text{Hilb}_{X/Y}$ is
separated over $Y$, $\iota_D$ is separated.  

\subsection{Infinitesimal study of the divisor restriction map} \label{ssec-inf}
\marpar{ssec-inf}

\mni
Let
$A'$ be a local Artin $\OO_Y$-algebra with maximal ideal
$\mathfrak{m}$ and residue field $\kappa$.  
And let 
$$
0 \rightarrow J \rightarrow A' \rightarrow A \rightarrow 0
$$
be an infinitesimal extension, i.e., $\mathfrak{m} J$ is zero.  Denote
by $X_{A'}$, resp. $X_A$, $X_\kappa$, the fiber product of
$X\rightarrow Y$ with $\SP A' \rightarrow Y$, resp. $\SP A \rightarrow
Y$, $\SP \kappa \rightarrow Y$.

\mni
Let
$(\mc{F}_{A'},u_{A'})$ be a pseudo-ideal sheaf of $D/Y$ over $\SP A'$.
Denote by $(\mc{F}_A,u_A)$, resp. $(\mc{F}_\kappa,u_\kappa)$, the
restriction of $(\mc{F}_{A'},u_{A'})$ to $A$, resp. to $\kappa$. 
Let $\mc{I}_A$ be the ideal sheaf of a flat family $C_A$ of closed
subschemes of $X/Y$ over $\SP A$.  Denote by $\mc{I}_\kappa$,
resp. $C_\kappa$, the restriction of $\mc{I}_A$ to $\kappa$, resp. of
$C_A$ to $\kappa$. 
And assume that $\iota_D$ sends
$\mc{I}_A$ to $(\mc{F}_A,u_A)$.

\begin{prop} \label{prop-inf}
\marpar{prop-inf}
Let $n$ be a nonnegative integer.
Assume that $C_\kappa$ is a regular immersion of codimension $n$ in
$X_\kappa$, 
cf. \cite[D\'efinition 16.9.2]{EGA4} (since $X_\kappa$ is an algebraic
space, in that definition 
one must replace the Zariski covering by affine schemes 
by an \'etale covering by affine schemes).     
\begin{enumerate}
\item[(i)]
The morphism $\iota_D$ is \emph{locally unobstructed} at $C_\kappa$ in
the following sense.  For every \'etale morphism $\SP R_{A'}
\rightarrow X_{A'}$ such that the pullback $I_\kappa$ of $\mc{I}_\kappa$ in
$R_\kappa$ is generated by a regular sequence, there exists an ideal
$I_{A'}$ in $R_{A'}$ whose restriction $I_A$ to $R_A$ equals the
pullback of $\mc{I}_A$ and whose ``local pseudo-ideal sheaf''
$(I_{A'}\otimes_{\OO_X} \OO_D,v)$ equals the pullback of
$(\mc{F}_{A'},u_{A'})$.   Moreover, the set of such ideals $I_{A'}$ is
naturally a torsor for the $R_\kappa$-submodule 
$$
J\otimes_\kappa \OO_X(-D)\cdot\text{Hom}_{R_\kappa}(I_\kappa,R_\kappa/I_\kappa)
$$
of
$$
J\otimes_\kappa \text{Hom}_{R_\kappa}(I_\kappa,R_\kappa/I_\kappa)
$$ 
(here
$\OO_X(-D)\cdot$ denote multiplication by the inverse image ideal of
$\OO_X(-D)$).  
\item[(ii)]
There exists an element $\omega$ in 
$$
J\otimes_\kappa H^1(C_\kappa,\OO_X(-D)\cdot
\textit{Hom}_{\OO_{C_\kappa}}(\mc{I}_\kappa/\mc{I}_\kappa^2,\OO_{C_\kappa}))
$$
which equals $0$ if and only if there exists a flat family $C_{A'}$ of
closed subschemes of $X/Y$ over $\SP A'$ whose restriction to $\SP A$
equals $C_A$ and whose image under $\iota_D$ equals
$(\mc{F}_{A'},u_{A'})$.  When it equals $0$, the set of such families
$C_{A'}$ is naturally a torsor for the $\kappa$-vector space
$$
J\otimes_\kappa H^0(C_\kappa,\OO_X(-D)\cdot
\textit{Hom}_{\OO_{C_\kappa}}(\mc{I}_\kappa/\mc{I}_\kappa^2,\OO_{C_\kappa})).
$$
\end{enumerate}
In particular, if $h^1(C_\kappa,\OO_X(-D)\cdot
\textit{Hom}_{\OO_{C_\kappa}}(\mc{I}_\kappa/\mc{I}_\kappa^2,\OO_{C_\kappa}))$
equals $0$, then $\iota_D$ is smooth at $[C_\kappa]$.  
\end{prop}

\begin{proof}
\textbf{(i)}  Let $v_{(A,1)},\dots,v_{(A,n)}$ be a regular sequence in
$R_A$ generating $I_A$.  Denote by $v_A$ the $R_A$-module homomorphism
$$
v_A:R_A^{\oplus n} \rightarrow R_A, \ \ v_A(\mathbf{e}_i) = v_{(A,i)}.
$$
Denote the associated $R_A$-module homomorphism by
$$
v'_A:R_A^{\oplus \binom{n}{2}} \rightarrow R_A^{\oplus n}, \ \
v'_A(\mathbf{e}_i\wedge \mathbf{e}_j) = v_{(A,i)}\mathbf{e}_j -
v_{(A,j)}\mathbf{e}_i.  
$$
Since $(v_{(A,1)},\dots,v_{(A,n)})$ is a regular sequence, the
following sequence is exact
$$
R_A^{\oplus \binom{n}{2}} \xrightarrow{v'_A} R_A^{\oplus n}
\rightarrow I_A \rightarrow 0.
$$
Thus there is an exact sequence
$$
(R_A\otimes_{\OO_X}\OO_D)^{\oplus \binom{n}{2}} \xrightarrow{v'_A\otimes
  \text{Id}} (R_A\otimes_{\OO_X}\OO_D)^{\oplus n} \rightarrow
(I_A\otimes_{\OO_X} \OO_D) \rightarrow 0.
$$

\mni
Denote by $F_{A'}$ the $R_{A'}$-module whose associated quasi-coherent
sheaf is the pullback of $\mc{F}_{A'}$.  Since $\mc{F}_A$ equals
$\mc{I}_A\otimes_{\OO_X} \OO_D$, there is an isomorphism
$$
F_{A'}/JF_{A'} \cong I_A\otimes_{\OO_X} \OO_D
$$
compatible with the maps $u_{A'}$ and $u_A$.  Since
$(R_{A'}\otimes_{\OO_X} \OO_D)^{\oplus
  n}$ is a projective $(R_{A'}\otimes_{\OO_X}\OO_D)$-module, 
there exists an $(R_{A'}\otimes_{\OO_X}\OO_D)$-module
homomorphism
$$
a:(R_{A'}\otimes_{\OO_X}\OO_D)^{\oplus n} \rightarrow F_{A'}
$$
whose restriction to $A$ is the surjection above.  So, by Nakayama's
lemma, this map is also surjective.  Since both the source and target
of the surjection are $A'$-flat, also
the kernel is $A'$-flat.  Thus there is also a lifting of the set of
generators of the kernel, i.e., there is an exact sequence of
$(R_{A'}\otimes_{\OO_X}\OO_D)$-modules
$$
(R_{A'}\otimes_{\OO_X}\OO_D)^{\oplus \binom{n}{2}} \xrightarrow{b}
(R_{A'}\otimes_{\OO_X}\OO_D)^{\oplus n} \xrightarrow{a} 
F_{A'} \rightarrow 0
$$
whose restriction to $A$ is the short exact sequence above.  

\mni
The composition of the surjection with $u_{A'}$ defines an 
$(R_{A'}\otimes_{\OO_X}\OO_D)$-module
homomorphism
$$
w_{A'}:(R_{A'}\otimes_{\OO_X}\OO_D)^{\oplus n} \rightarrow
(R_{A'}\otimes_{\OO_X}\OO_D)
$$
whose restriction to $A$ equals $v_A\otimes \text{Id}$.  There is an
associated map
$$
w'_{A'}:(R_{A'}\otimes_{\OO_X}\OO_D)^{\oplus \binom{n}{2}} \rightarrow
(R_{A'}\otimes_{\OO_X}\OO_D)^{\oplus n}, \ \
w'_{A'}(\mathbf{e}_i\wedge \mathbf{e}_j) =
w_{A'}(\mathbf{e}_i)\mathbf{e}_j - w_{A'}(\mathbf{e}_j)\mathbf{e}_i.
$$
Since
$(\mc{F}_{A'},u_{A'})$ is a pseudo-ideal sheaf, the induced map
$u'_{A'}$ equals $0$.  Therefore the image of $w'_{A'}$ is contained
in the kernel of $a$.  Since $(R_{A'}\otimes_{\OO_X} \OO_D)^{\oplus
  \binom{n}{2}}$ is a free $(R_{A'}\otimes_{\OO_X} \OO_D)$-module,
there is a lifting
$$
c:(R_{A'}\otimes_{\OO_X}\OO_D)^{\oplus \binom{n}{2}} \rightarrow
(R_{A'}\otimes_{\OO_X}\OO_D)^{\oplus \binom{n}{2}} 
$$
such that $w'_{A'} = b\circ c$.  In particular, the restriction of $c$
to $A$ is an isomorphism.  Thus, by Nakayama's lemma, $c$ is
surjective.  A surjection of free modules of the same finite rank is
automatically an isomorphism.  Thus there is a presentation
$$
(R_{A'}\otimes_{\OO_X}\OO_D)^{\oplus \binom{n}{2}} \xrightarrow{w'_{A'}}
(R_{A'}\otimes_{\OO_X}\OO_D)^{\oplus n} \xrightarrow{a} 
F_{A'} \rightarrow 0.
$$

\mni
Because both $R_{A'}$ and $R_{A'}\otimes_{\OO_X}\OO_D$ are $A'$-flat,
there is a commutative diagram of exact sequences
$$
\begin{CD}
0 @>>> J\otimes_{\kappa} R_\kappa @>>> R_{A'} @>>> R_A @>>> 0 \\
& &  @VVV  @VVV  @VVV \\
0 @>>> J\otimes_\kappa R_\kappa \otimes_{\OO_X} \OO_D @>>>
R_{A'}\otimes_{\OO_X} \OO_D @>>> R_A \otimes_{\OO_X} \OO_D @>>> 0
\end{CD}
$$
where the vertical maps are each surjective.  By the snake lemma, the
induced map
$$
R_{A'}/(J\OO_X(-D)\cdot R_{A'}) \rightarrow R_A
\times_{(R_A\otimes_{\OO_X} \OO_D)} ( R_{A'}\otimes_{\OO_X} \OO_D )
$$
is an isomorphism.  
Thus, for every integer $i=1,\dots,n$, there exists an element
$v_{(A',i)}$ in $R_{A'}$ whose image in $R_A$ equals $v_{(A,i)}$ and
whose image in $R_{A'}\otimes_{\OO_X}\OO_D$ equals
$w_{A'}(\mathbf{e}_i)$.  Moreover, the set of all such elements is
naturally a torsor for $J\otimes_\kappa (\OO_X(-D)\cdot R_\kappa)$.
In other words, there is an $R_{A'}$-module homomorphism
$$
v_{A'}:R_{A'}^{\oplus n} \rightarrow R_{A'}
$$
whose restriction to $A$ equals $v_A$ and such that $v_{A'}\otimes
\text{Id}$ equals $w_{A'}$.  

\mni
Since $(v_{(A,1)},\dots,v_{(A,n)})$ is a
regular sequence in $R_A$, also $(v_{(A',1)},\dots,v_{(A',n)})$ is a
regular sequence in $R_{A'}$. One way to see this is to tensor the
Koszul complex $K^\bullet(R_{A'},v_{A'})$ of $v_{A'}$ with the short exact
sequence of $A'$-modules
$$
0 \rightarrow J \rightarrow A' \rightarrow A \rightarrow 0.
$$
Since the terms in the Koszul complex are free $R_{A'}$-modules, and
since $R_{A'}$ is $A'$-flat, the associated sequence of complexes is
exact
$$
0 \rightarrow J\otimes_A K^\bullet(R_A,v_A) \rightarrow
K^\bullet(R_{A'},v_{A'}) \rightarrow K^\bullet(R_A,v_A) \rightarrow 0.
$$
Thus there is a 
long exact sequence of Koszul cohomology
$$
\dots \rightarrow J\otimes_A H^n(K^\bullet(R_A,v_{A})) \rightarrow
H^n(K^\bullet(R_{A'}, v_{A'})) \rightarrow H^n(K^\bullet(R_A,v_A))
\rightarrow J\otimes_A H^{n+1}(K^\bullet(R_A,v_A)) \rightarrow \dots
$$
Since $v_A$ is regular, $H^{n-1}(K^\bullet(R_A,v_A))$ is zero, which
then implies $H^{n-1}(K^\bullet(R_{A'},v_{A'}))$ by the long exact
sequence above.  Thus also $v_{A'}$ is regular.  Moreover, this gives
a short exact sequence
$$
0 \rightarrow J\otimes_\kappa (R_A/I_A) \rightarrow
R_A'/\text{Image}(v_{A'}) \rightarrow R_A/I_A \rightarrow 0
$$
from which it follows that $R_{A'}/\text{Image}(v_{A'})$ is
$A'$-flat.  Denote by $I_{A'}$ the image of $v_{A'}$.  

\mni
Both the pseudo-ideal $I_{A'}\otimes_{\OO_X}\OO_D$ and $F_{A'}$ equal
the cokernel of $w'_{A'}$.  Thus there is a unique isomorphism between
them compatible with $w'_{A'}$.  Moreover, since the compositions
$$
(R_{A'}\otimes_{\OO_X} \OO_D)^{\oplus n} \xrightarrow{a} F_A'
\xrightarrow{u_{A'}} (R_{A'}\otimes_{\OO_X} \OO_D)
$$
and
$$
(R_{A'}\otimes_{\OO_X} \OO_D)^{\oplus n} \xrightarrow{v_{A'}\otimes
  \text{Id}} I_A'\otimes_{\OO_X} \OO_D
\rightarrow (R_{A'}\otimes_{\OO_X} \OO_D)
$$ 
both equal $w_{A'}$, the isomorphism above is compatible with the maps
to $(R_{A'}\otimes_{\OO_X} \OO_D)$.  Thus it is an isomorphism of
pseudo-ideal sheaves.  Therefore there exists
$I_{A'}$ satisfying all the conditions
in the proposition.

\mni
For every lift $I_{A'}$ of $I_A$ which is $A$-flat, the surjection
$v_A:R_A^{\oplus n} \rightarrow I_A$ lifts to a surjection
$R_{A'}^{\oplus n} \rightarrow  I_{A'}$.  Composing this surjection
with the injection $I_{A'} \hookrightarrow R_{A'}$, it follows that
every lift $I_{A'}$ arises from a lift $v_{A'}$ of $v_A$.  As
mentioned previously, the set of lifts $v_{A'}$ whose restriction to $A$
equals $v_A$ and with $v_{A'}\otimes \text{Id}$ equal to $w_{A'}$ is
naturally a torsor for 
$$
J\otimes_\kappa \text{Hom}_{R_{A'}}(R_{A'}^{\oplus n},\OO_X(-D)\cdot
R_\kappa).
$$
But the translate of a lift by a homomorphism with image in
$I_\kappa$ gives the same ideal $I_{A'}$ (just different surjections
from $R_{A'}^{\oplus n}$ to the ideal).  Thus, the set of lifts
$I_{A'}$ of $I_A$ whose pseudo-ideal sheaf is the pullback of
$(\mc{F}_{A'},u_{A'})$ is naturally a torsor for
$$
J\otimes_\kappa \text{Hom}_{R_{A'}}(R_{A'}^{\oplus n},\OO_X(-D)\cdot
(R_\kappa/I_\kappa)) = J\otimes_\kappa
\text{Hom}_{R_\kappa}(I_\kappa,R_\kappa/I_\kappa).
$$

\mni
\textbf{(ii)}
Let $\Gamma$ be an indexing set and let $(\SP R^\gamma_{A'}
\rightarrow X_{A'})_{\gamma \in \Gamma}$ be an \'etale covering such
that for every $\gamma$, either $I^\gamma_\kappa \subset
R^\gamma_\kappa$ is generated by a regular sequence of length $n$ or
else equals $R^\gamma_\kappa$.  By the hypothesis that $C_\kappa$ is a
regular immersion of codimension $n$, there exists such a covering.  

\mni
Suppose first that $I^\gamma_\kappa$ equals $R^\gamma_\kappa$.  Then
also $(F^\gamma_\kappa,u^\gamma_\kappa)$ is isomorphic to $R^\gamma_\kappa
\otimes_{\OO_X} \OO_D \xrightarrow{=} R^\gamma_\kappa \otimes_{\OO_X}
\OO_D$.  It is straightforward to see that the only deformations over
$A'$
of
$R^\gamma_\kappa \xrightarrow{=} R^\gamma_\kappa$, resp. 
$R^\gamma_\kappa \otimes_{\OO_X} \OO_D \xrightarrow{=} R^\gamma_\kappa
\otimes_{\OO_X} 
\OO_D$, as pseudo-ideals are $R^\gamma_{A'} \xrightarrow{=} R^\gamma_{A'}$,
resp. $R^\gamma_{A'} \otimes_{\OO_X} \OO_D \xrightarrow{=}
R^\gamma_{A'}\otimes_{\OO_X} \OO_D$.  Thus there is a lifting
$I^\gamma_{A'}$ of $I^\gamma_A$, and it is unique.

\mni
On the other hand, if $I^\gamma_\kappa$ is generated by a regular
sequence of length $n$, by (i) there exist liftings $I^\gamma_{A'}$
and the set of all liftings is a torsor for 
$$
J\otimes_\kappa
\OO_X(-D)\cdot\text{Hom}_{R^\gamma_\kappa}(I^\gamma_\kappa ,
R^\gamma_\kappa/I^\gamma_\kappa).
$$
For a collection of liftings $(I^\gamma_{A'})_{\gamma \in \Gamma}$,
for every $\gamma_1,\gamma_2 \in \Gamma$, the basechanges of
$I^{\gamma_1}_{A'}$ and $I^{\gamma_2}_{A'}$ to
$$
R^{\gamma_1,\gamma_2}_{A'}:= R^{\gamma_1}_{A'}\otimes_{\OO_X}
R^{\gamma_2}_{A'}
$$ 
differ by an element $\omega^{\gamma_1,\gamma_2}$ in 
$$
J\otimes_\kappa
\OO_X(-D)\cdot\text{Hom}_{R^{\gamma_1,\gamma_2}_\kappa}(I^{\gamma_1,\gamma_2}_\kappa
, R^{\gamma_1,\gamma_2}_\kappa/I^\gamma_\kappa).
$$
It is straightforward to see that
$(\omega^{\gamma_1,\gamma_2})_{\gamma_1,\gamma_2 \in \Gamma}$ is a
$1$-cocycle for 
$$
\OO_X(-D)\cdot
\textit{Hom}_{\OO_{C_\kappa}}(\mc{I}_\kappa/\mc{I}_\kappa^2,\OO_{C_\kappa})
$$
with respect to the given \'etale covering.  Moreover, changing the
collection of lifts $(I^\gamma_{A'})$ by translating by elements in 
$$
(J\otimes_\kappa
\OO_X(-D)\cdot\text{Hom}_{R^\gamma_\kappa}(I^\gamma_\kappa ,
R^\gamma_\kappa/I^\gamma_\kappa))_{\gamma \in \Gamma}
$$
precisely changes $(\omega^{\gamma_1,\gamma_2})_{\gamma_1,\gamma_2\in
  \Gamma}$ by a $1$-coboundary.  Therefore, the cohomology class
$$
\omega \in  J\otimes_\kappa H^1(C_\kappa,\OO_X(-D)\cdot
\textit{Hom}_{\OO_{C_\kappa}}(\mc{I}_\kappa/\mc{I}_\kappa^2,\OO_{C_\kappa}))
$$
is well-defined and equals $0$ if and only if there is a lifting
$\mc{I}_{A'}$ as in the proposition.  And in this case, the set of
liftings is a torsor for the set of compatible families
$(t^\gamma)_{\gamma\in \Gamma}$ of elements $t^\gamma$ in 
$$
J\otimes_\kappa
\OO_X(-D)\cdot\text{Hom}_{R^\gamma_\kappa}(I^\gamma_\kappa ,
R^\gamma_\kappa/I^\gamma_\kappa),
$$
i.e., it is a torsor for the set of elements $t$ in
$$
J\otimes_\kappa H^0(C_\kappa,\OO_X(-D)\cdot
\textit{Hom}_{\OO_{C_\kappa}}(\mc{I}_\kappa/\mc{I}_\kappa^2,\OO_{C_\kappa})).
$$
\end{proof}


\section{Deforming comb-like curves} \label{sec-def}
\marpar{sec-def}

\mni
Let $k$ be an algebraically closed field.  Let $B$ be a smooth,
projective, connected curve over $k$.  Let $X$ be a normal, 
projective, connected $k$-scheme.  And let $\pi:X\rightarrow B$ be a
surjective $k$-morphism of relative dimension $d$.  
Assume that the geometric generic fiber of
$\pi$ is normal and contains a very free rational curve inside its
smooth locus.  

\begin{lem} \label{lem-comblike}
\marpar{lem-comblike}
There is an open subscheme of $\text{Hilb}_{X/k}$ precisely
parameterizing the comb-like curves.
\end{lem}

\begin{proof}
There is an open subscheme of $\text{Hilb}_{X/k}$ parameterizing
regular immersions of codimension $d-1$.  The degree of such a curve
over $B$ is a locally constant integer-valued function.  Thus the open
subscheme of comb-like curves is precisely the open subset on which
the degree equals $1$.
\end{proof}

\begin{prop} \label{prop-comblike}
\marpar{prop-comblike}
Let $s:B\rightarrow X$ be a section of $\pi$ 
mapping the geometric generic point of $B$ into the very
free locus of the geometric generic fiber of $\pi$.  
Let $E$ be an effective Cartier divisor in $B$, and denote by $D$ the
inverse image 
Cartier divisor $\pi^{-1}(E)$ in $X$.  Let $C$ be a comb-like curve
with handle $s$ which is contained in the scheme-theoretic union of
$s(B)$ and $D$.  There exists an integer $N$ with the following
property.  For every comb-like curve $C'$ obtained by attaching at
least $N$ very free curves in fibers of $\pi$ to $C$ at general points
of $s(B)$ and with general normal directions to $s(B)$, the divisor 
restriction map, $\iota_D:\text{Hilb}_{X/k} \rightarrow \pis{D/k}$, is
smooth at $[C']$.  
\end{prop}

\begin{proof}
Since $s$ maps the geometric generic point of $B$ into the very free
locus of $\pi$, with the exception of finitely many points, for every
$k$-point $b$ of $B$, $\pi$ is smooth at $s(b)$ and for every normal
direction to $s(B)$ in $X$ at $s(b)$, there exists a very free curve in
$\pi^{-1}(b)$ containing $s(b)$ and having the specified normal
direction.   
So long as the attachment points $s(b)$ is a smooth point of $C$
(which also holds with finitely many exceptions), the curve obtained
by attaching this very free curve to $C$ is again comb-like.  The
final exception is that we only attach very free curves at points $b$
not contained in $E$.  

\mni
Let $C'$ be a curve obtained from $C$ by attaching 
a collection of very free rational curves in fibers of $\pi$ at points
of $s(B)$ which are smooth points of $C$. 
Denote by $C''$ the subcurve of $C'$ consisting of $s(B)$ and all the
attached very free curves in fibers of $\pi$.  
Since
$C'\rightarrow X$ is a regular immersion, the sheaf
$\mc{I}_{C'/X}/\mc{I}_{C'/X}^2$ is a locally free $\OO_{C'}$-module,
where $\mc{I}_{C'/X}$ is the ideal sheaf of $C'$ in $X$.  Denote by
$N_{C'/X}$ the dual sheaf, 
$$
N_{C'/X} := \textit{Hom}_{\OO_{C'}}( \mc{I}_{C'/X}/\mc{I}_{C'/X}^2 ,
\OO_{C'} ).
$$
The subsheaf $\OO_X(-D)\cdot N_{C'/X}$ is supported on $C''$ since $C$
is contained in the scheme theoretic union of $s(B)$ and $D$.
Moreover, $\OO_X(-D)\cdot N_{C'/X}|_{s(B)}$ is obtained from
$\OO_X(-D)\cdot N_{C/X}|_{s(B)}$ by the construction from \cite[Lemma
2.6]{GHS}.  Thus, by the argument from \cite[p. 62]{GHS}, if $C'$ is
obtained from $C$ by attaching sufficiently many very free curves in
fibers of $\pi$ to general points of $s(B)$ with general normal
directions, then $h^1(C'', \OO_X(-D)\cdot N_{C'/X})$ equals $0$.
Thus, by Proposition ~\ref{prop-inf}, the divisor restriction map is
smooth at $[C']$.  
\end{proof}


\begin{cor} \label{cor-comblike}
\marpar{cor-comblike}
With the same hypotheses as in Proposition ~\ref{prop-comblike},
let $F$ be a subdivisor of $E$, let $s_F$ be a section of $\pi$ over $F$,
assume that the pseudo-ideal sheaf $\iota_D([C])$ deforms to
pseudo-ideal sheaves agreeing with $s_F$.  Then there exists a section
$s_\infty$ 
of $\pi$ whose restriction to $F$ equals $s_F$.  
\end{cor}

\begin{proof}
Since none of the very free curves attached to $C$ intersect $D$,
$\iota_D([C])$ equals $\iota_D([C'])$.  Let $(T,t)$ be as in
Definition ~\ref{defn-comblike}.  Form the pullback morphism
$$
\iota_{D,T}:T\times_{\psi{D/k},\iota_D} \text{Hilb}_{X/k} \rightarrow
T.
$$
By Proposition ~\ref{prop-comblike}, this morphism is smooth at
$(t,[C'])$.  By Lemma
~\ref{lem-comblike}, there is an open neighborhood $U$ of $(t,[C'])$
precisely parameterizing pairs for which the closed subscheme is a
comb-like curve.  The restriction of $\iota_D$ to this open subscheme
is also smooth at $(t,[C'])$.  Thus it is dominant.  So there exists a
$k$-point $(t_\infty,[C_\infty])$
in this open subscheme such that $t_\infty \neq t$.  Therefore the
pseudo-ideal sheaf of $C_\infty$ agrees with $s_F$.  Let $s_\infty$ be
the unique section of $\pi$ such that $s_\infty(B)$ is an irreducible
component of $C_\infty$.  Then the ideal sheaf of $s_\infty(E_F)$ is a
subsheaf of the pseudo-ideal sheaf of $C_\infty$ over $E_F$.   But two
ideal sheaves of sections can be contained one in the other if and
only if they are equal.  Therefore the restriction of $s_\infty$ to
$E_F$ is the given section over $E_F$.  In particular, the restriction
of $s_\infty$ to $F$ equals $s_F$.1
\end{proof}


\section{Proofs of the main theorems} \label{sec-proofs}
\marpar{sec-proofs}

\begin{proof}[Proof of Theorem ~\ref{thm-main}]
Let $E'$ be an effective divisor in $B$ containing $E$ and such that
for the inverse image Cartier divisor $D=\pi^{-1}(E')$, $C$ is
contained in the scheme-theoretic union $s(B) \cup D$.  The \'etale
local deformation of $C$ to section curves agreeing with $s_E$ in
particularly restricts on $D$ to a deformation of the pseudo-ideal
sheaf of $C$ to ideal sheaves of section curves agreeing with $s_E$.
Thus the theorem follows from Corollary ~\ref{cor-comblike}.  
\end{proof}

\begin{proof}[Proof of Proposition ~\ref{prop-Hassett}]
First we consider the case when the $\OO$-morphism
$$
f:\PP^1_K \rightarrow \SP \OO \times_B X
$$
does not necessarily extend to all of $\PP^1_\OO$.  

\begin{claim} \label{claim-Hassett1}
\marpar{claim-Hassett1}
There exists an
$E$-conic $\nu:P\rightarrow \PP^1\times_k B$ such that the composition
$$
f\circ \nu: \SP K \times_B P \rightarrow \PP^1_K \rightarrow \SP \OO
\times_B X
$$
extends to an $\OO$-morphism
$$
f_P: \SP \OO\times_B P \rightarrow \SP \OO \times_B X.
$$
\end{claim}

\mni
To prove this, first form the closure $\Sigma$
of the graph of $f$ in $\PP^1_\OO
\times_B X$.  This gives a possibly singular, $2$-dimensional scheme
whose projection onto $\PP^1_\OO$ is projective and birational.  By
resolution of 
singularities for $2$-dimensional schemes, e.g. as in \cite{Lipman},
there exists 
a projective desingularization of $\Sigma$.  The composition with
projection onto $\PP^1_\OO$ is a projective, birational morphism, thus
equals the blowing up of $\PP^1_\OO$
along a coherent ideal sheaf $\mc{I}_\OO$
supported on the closed fiber.  Since
every such ideal sheaf $\mc{I}_\OO$ equals the basechange of a coherent ideal
sheaf $\mc{I}$ on $\PP^1\times_k B$ supported on $\PP^1\times_k \{b\}$, the
blowing up of $\PP^1_\OO$ along $\mc{I}_\OO$
equals $\SP \OO \times_B P$ where $P$ is the
$E$-conic obtained by blowing up $\PP^1\times_k B$ along $\mc{I}$.
This proves Claim ~\ref{claim-Hassett1}.

\begin{claim} \label{claim-Hassett2}
\marpar{claim-Hassett2}
Denote the graph of $f_P$ by
$$
\Gamma_{f_P}:\SP \OO \times_B P \rightarrow \SP \OO \times_B
(X\times_B P).
$$
Each comb-like curve $C$ in $P$ which equals $s(B)$ over
$B-\{b\}$ equals the isomorphic projection of a unique 
comb-like curve $C_\Gamma$ in $X\times_B P$ 
whose basechange to $\SP \OO \times_B
(X\times_B P)$ equals $\Gamma_{f_P}(\SP \OO\times_B C)$.
\end{claim}

\mni
There exists a nonnegative integer $r$ such that $C$ is contained in
the scheme-theoretic union of $s(B)$ and $(r\cdot b)\times_B P$.
Denote by $Z$ the scheme-theoretic union of $s(B)$ and $(r\cdot
b)\times_B (X\times_B P)$.  The pullback of $\OO_Z$ to $\SP \OO
\times_B (X\times_B P)$ equals the scheme-theoretic union of the germ
of $s(B)$ and $(r\cdot b)\times_B (X\times_B P)$.  The ideal sheaf
$\mc{I}$
of $\Gamma_{f_P}(\SP \OO\times_B C)$ in $\SP \OO \times_B Z$ is a
coherent subsheaf supported entirely on $(r\cdot b)\times_B (X\times_B
P)$.  Thus there exists a unique coherent ideal sheaf
$\mc{I}_{C_\Gamma}$ in $\OO_Z$ supported entirely on $(r\cdot
b)\times_B (X\times_B P)$ whose pullback to $\SP \OO \times_B
(X\times_B P)$ equals $\mc{I}$.  The corresponding closed subscheme
$C_\Gamma$ of $Z$ is the unique closed subscheme whose basechange to
$\SP \OO \times_B (X\times_B P)$ equals $\Gamma_{f_P}(\SP \OO\times_B
C)$.  Moreover, because the projection of $\Gamma_{f_P}(\SP \OO
\times_B C)$ into $\SP \OO\times_B P$ equals $\SP \OO \times_B C$,
also the projection of $C_\Gamma$ into $P$ equals $C$.

\mni
Finally, since $\SP \OO\times_B X$ is smooth along $S$, also $\SP \OO
\times_B (X\times_B P)$ is smooth along $\Gamma_{f_P}(\SP \OO \times_B
P)$.  Thus $\Gamma_{f_P}$ is a regular immersion.  Since $\SP
\OO\times_B C$ is a Cartier divisor in $\SP \OO \times_B P$, also
$\Gamma_{f_P}(C)$ is regularly immersed in $\SP \OO \times_B P$.  Thus
$C_\Gamma$ is regularly immersed \'etale locally, resp. formally
locally, near $\{b\} \times_B (X\times_B P)$.  Now the property of
being a regular immersion is a formal local property since the Koszul
cohomology relative to the complete local ring of the images of a
minimal set of generators of the stalk of the ideal equals the
completion of the Koszul cohomology relative to the local ring of the
minimal set of generators.  Thus $C_\Gamma$ is a regular immersion at
every point of $C_\Gamma$ in $\{b\}\times_B (X\times_B P)$.  Since
also
$X\times_B P$ is smooth along $(s,0_P)$, $C_\Gamma$ is regularly
immersed at every point of $C_\Gamma$, i.e., $C_\Gamma$ is a comb-like
curve in $X\times_B P$.  This proves Claim ~\ref{claim-Hassett2}. 

\begin{claim} \label{claim-Hassett3}
\marpar{claim-Hassett3}
If there exists a comb-like curve $C$ in $P$ with handle $0_P$ which
\'etale locally deforms to section curves agreeing with the
restriction $\infty_P|_E$, then there exists a comb-like curve $C$ in
$X\times_B P$ with handle $(s,0_P)$ which \'etale locally deforms to
section curves agreeing with an $E$-section $(s_E,t_E)$.
\end{claim}

\mni
There is a unique $E$-section of $X\times_B P$ such that
$\Gamma_{f_P}\circ \infty_P$ equals $(s_E,t_E)$.  
By Claim ~\ref{claim-Hassett2}, there exists a comb-like curve
$C_\Gamma$ in $X\times_B P$ with handle $(s,0_P)$ whose basechange to
$\SP \OO \times_B (X\times_B P)$ equals $\Gamma_{f_P}(\SP \OO \times_B
C)$.  If $\SP \OO \times_B C$ deforms to a family of section curves $D_t$ in
$\SP \OO \times_B P$ agreeing with $\infty_P|_E$, then
$\Gamma_{f_P}(\SP \OO \times_B C)$ deforms to the family of section
curves $\Gamma_{f_P}(D_t)$ in $\SP \OO \times_B (X\times_B P)$
agreeing with $\Gamma_{f_P}\circ \infty_P = (s_E,t_E)$.  This proves
Claim ~\ref{claim-Hassett3}.

\begin{claim} \label{claim-Hassett4}
\marpar{claim-Hassett4}
There exists a comb-like curve $C$ in $P$ with handle $0_P$ which
Zariski locally deforms to section curves in $P$ agreeing with $\infty_P|_E$.  
\end{claim}

\mni
On $P$
consider the invertible sheaf $\mc{L} = \OO_P(\infty_P(B) - 0_P(B) -
\pi^*E)$. Since the restriction to the generic fiber of $P\rightarrow
B$ is a degree $0$ invertible sheaf on $\PP^1$, it has a
$1$-dimensional space of global sections.  Thus the pushforward of
$\mc{L}$ to $B$ is a torsion-free, coherent $\OO_B$-module of rank
$1$, i.e., it is an invertible sheaf.  Thus there exists an effective
divisor $\Delta$ in $B$, not intersecting $E$ such that
$\OO_P(\infty_P(B) - 0_P(B) +\pi^* \Delta - \pi^* E)$ has a global section.
In other words, there is an effective divisor $B$ in $P$, necessarily
vertical, such that
$$
\infty_P(B) + \pi^* \Delta = 0_P(B) + \pi^* E + B.
$$
Define $C$ to be $0_P(B) + \pi^* E + B$ and let the curves $D_t$ be
the members of the pencil spanned by $C$ and $\infty_P(B) +
\pi^*\Delta$.  Since $\infty_P(B)+\pi^*\Delta$ is a section curve over
$E$, all but finitely many members of this pencil are section curves
over $E$.  Since the base locus of the pencil contains
$\infty_P(B)\cap \pi^* E$, these section curves agree with
$\infty_P(B)$ over $E$.  This proves Claim ~\ref{claim-Hassett4},
completing the proof of the proposition.  

\mni
Of course if $f$ extends to a closed immersion
$$
f_\OO:\SP \OO\times_k \PP^1 \rightarrow \SP \OO \times_B X,
$$
then one can take $P$ to be $\PP^1\times_k B$ and one can work with
$f_\OO$ instead of $\Gamma_{f_P}$.  In this way one produces a
comb-like curve in $X$ which \'etale locally, resp. formally locally,
deforms to section curves in $X$ agreeing with $s_E$.  
\end{proof}

\begin{proof}[Proof of Corollary ~\ref{cor-HT}]
First of all, the results of the first paragraph follow from the
results of the second, since every pair of points in a smooth,
projective, separably rationally connected variety are connected by a
very free rational curve, which may also be taken to be immersed if
the dimension is $\geq 3$.  Next, if there is a very free rational
curve in $\pi^{-1}(b)$
containing $s(b)$ and $s_E(b)$, then deformations of the
rational curve in fibers of $\pi$ and intersecting both $s$ and $s_E$
are unobstructed.  Thus there exist both \'etale local and formal
local morphisms
$$
f_\OO:\SP \OO \times_k \PP^1 \rightarrow \SP \OO \times_B X
$$
whose closed fiber is the given very free rational curve in
$\pi^{-1}(b)$ and whose $0$-section, resp. $\infty$-section, is the
basechange of $s$, resp. agrees with $s_E$.  If the original very free
rational curve in $\pi^{-1}(b)$ is an immersion, then also $f_\OO$ is an
immersion.  So the corollary follows from Theorem ~\ref{thm-main} and
Proposition ~\ref{prop-Hassett}. 
\end{proof}

\begin{proof}[Proof of Corollary ~\ref{cor-rational}]
Let $U$ be a dense Zariski open subset of $\PP^n_{\SP
  \widehat{K}_{B,b}}$ 
which is isomorphic to a dense Zariski open subset of
$\SP \widehat{K}_{B,b} \times_B( X\times_k \PP^r_k)$.  Possibly after
deforming $s$, we may assume the $\widehat{K}_{B,b}$-point of $\SP
  \widehat{K}_{B,b}\times_B X$ determined by 
$s$ is
the image of a $\widehat{K}_{B,b}$-point of $\SP \widehat{K}_{B,b}
  \times_B( X\times_k \PP^r_k)$ which is
contained in $V$.
Similarly, there exists a formal section of $\SP \widehat{\OO}_{B,b}
  \times_B X$ agreeing with $s_E$ whose $\widehat{K}_{B,b}$-point is
the image of a
$\widehat{K}_{B,b}$-point of $V$.  The isomorphism of $U$ and $V$ sends
these two $\widehat{K}_{B,b}$-points of $V$ two $\widehat{K}_{B,b}$
points of $U$.  Every pair of $\widehat{K}_{B,b}$-points of
  $\PP^n_{\SP \widehat{K}_{B,b}}$ are contained in a line which is
  isomorphic to $\PP^1_{\SP \widehat{K}_{B,b}}$.  The composition with
  the isomorphism from $U$ to $V$ determines a $\widehat{K}_{B,b}$-rational map
$$
f_K:\PP^1_{\SP \widehat{K}_{B,b}} \dashrightarrow 
\SP \widehat{K}_{B,b} \times_B( X\times_k \PP^r_k)
$$
which extends to all of $\PP^1_{\SP \widehat{K}_{B,b}}$ by the
valuative criterion of properness.  Now apply Theorem ~\ref{thm-main}
and Proposition ~\ref{prop-Hassett}.  
\end{proof}

\begin{proof}[Proof of Corollary ~\ref{cor-cubic}]
Let $b$ be a point of $B$ for which $X_b$ is a cubic surface with
rational double point singularities.  The Galois group of
$\widehat{K}_{B,b}$ is topologically cyclic.  The action on the
N\'eron-Severi lattice of the geometric generic fiber
$\SP \overline{\widehat{K}}_{B,b} \times_B X$ sends a topological
generator to an element in the automorphism group of the
N\'eron-Severi lattice, namely $\mc{W}(E_6)$.  The conjugacy class of
this element in $\mc{W}(E_6)$ is well-defined independent of the
choice of topological generator.  We call it the \emph{monodromy class}.

\mni
The possibilities for the monodromy class
are listed in
\cite[Table 1, p. 176]{CubicForms}.  As explained in \cite[31.1(2),
p. 174]{CubicForms}, the cubic surface $\SP \widehat{K}_{B,b} \times_B
X$ is minimal if and only if the monodromy class is in one of the
first $5$ rows of the table, and the minimal model is a Del Pezzo of
degree $\geq 5$ unless it is one of the first $11$ rows.  Since every
smooth Del Pezzo of degree $\geq 5$ is rational, Corollary
~\ref{cor-rational} applies unless the monodromy class is in one of
the first 11 rows.

\begin{claim} \label{claim-cubic1}
\marpar{claim-cubic1}
If the cubic surface $X_b$ has only rational double points, and no
greater than $3$ of these, then the monodromy class is not in one of
the first $11$ rows.
\end{claim}

\mni
The $8^\text{th}$ column of each row lists the orbit decomposition of
the $27$ lines.  Each orbit specializes to a line in the singular fiber
whose multiplicity is at least as large as the size of the orbit.
More precisely, the multiplicity of each line in the singular fiber
equals the sum of the
lengths of the orbits specializing to that line.

\mni
Every line in the singular fiber not containing a singular point
necessarily has multiplicity $1$ since the normal bundle of a line in
the nonsingular locus of a cubic surface is isomorphic to 
the nonspecial invertible
sheaf $\OO_{\PP^1}(-1)$. 
Thus it remains to consider the multiplicities
of lines containing a rational double point.  For each double point,
the
length of the scheme
of lines in the cubic surface containing that double point is always
$6$.  The multiplicity of a line joining two or more 
double points equals the sum of the contributions from each double
point it contains.
Thus, if
there are $m$ rational double points, then the sum of the multiplicities of
the lines containing at least one double point equals $6m$.  Since all
other lines have multiplicity $1$, the
number of orbits of size $1$ must be at least $27-6m$.  Since every
orbit type in the first 11 rows begins with $1^e$ for $e\leq 3$, such
a monodromy class may occur only if the cubic surface has $4$ double
points.  This proves Claim ~\ref{claim-cubic1}.  Finally apply
Corollary ~\ref{cor-rational} to complete the proof.
\end{proof}

\bibliography{my}
\bibliographystyle{alpha}

\end{document}